\documentclass[12pt]{amsart}
\usepackage{amsmath, amssymb, amsthm, amsfonts, mathrsfs}
\usepackage{cite}
\usepackage{amscd}
\usepackage{url}
\usepackage{graphicx}
\usepackage{caption}
\usepackage{subcaption}
\usepackage{fullpage}

\newtheorem{proposition}{Proposition}

\newtheorem{theorem}{Theorem}
\newtheorem{corollary}[proposition]{Corollary}
\theoremstyle{definition}

\newcommand{\incentive}{\varphi}

\makeatletter
\g@addto@macro{\endabstract}{\@setabstract}
\newcommand{\authorfootnotes}{\renewcommand\thefootnote{\@fnsymbol\c@footnote}}%
\makeatother
\begin{document}
\date{}

\begin{center}
  \LARGE 
  
    Entropy Rates of the Multidimensional Moran Processes and Generalizations \bigskip

  \normalsize
  \authorfootnotes
  Marc Harper\footnote{Affiliation where the bulk of the work was done; email: marc.harper@gmail.com}\textsuperscript{1}

  \textsuperscript{1}Department of Genomics and Proteomics, UCLA \par
  
\end{center}

\begin{abstract}
The interrelationships of the fundamental biological processes natural selection, mutation, and stochastic drift are quantified by the entropy rate of Moran processes with mutation, measuring the long-run variation of a Markov process. The entropy rate is shown to behave intuitively with respect to evolutionary parameters such as monotonicity with respect to mutation probability (for the neutral landscape), relative fitness, and strength of selection. Strict upper bounds, depending only on the number of replicating types, for the entropy rate are given and the neutral fitness landscape attains the maximum in the large population limit.  Various additional limits are computed including small mutation, weak and strong selection, and large population holding the other parameters constant, revealing the individual contributions and dependences of each evolutionary parameter on the long-run outcomes of the processes.
\end{abstract}


\section{Introduction}

Populations of replicating entities are subject to a variety of selective, stochastic, and diversifying processes \cite{burger1994distribution} \cite{bulmer1991selection}. Variation in population dynamics arises from both short-term and long-term effects, including short-term stochastic drift in small populations to long-term selective processes that slowly lead to the fixation of one trait in a population. These process have been studied extensively in models of finite populations and applications \cite{moran1958random} \cite{kimura1985neutral} \cite{fogel1998instability}  \cite{ficici2000effects} \cite{nowak2004emergence} \cite{taylor2004evolutionary} \cite{imhof2006evolutionary} \cite{traulsen2006stochasticity} \cite{antal2006fixation} \cite{dingli2011stochastic}. Recent work has begun to understand the relationships of these fundamental processes of replicative proliferation, drift, and variation using entropic methods \cite{andrae2010entropy} \cite{harper2013inherent} \cite{hernando2013workings} (see also \cite{wakeley2005limits}).

Much of the study of models of evolutionary processes is focused on long-term outcomes, such as fixation of one or more types, the stability of population states, and the stationary distributions of probabilistic models \cite{antal2006fixation} \cite{taylor2004evolutionary} \cite{imhof2006evolutionary} \cite{claussen2005non} \cite{traulsen2009stochastic} \cite{akin1984evolutionary} \cite{harper2012stability}. In this paper we quantity the variation in long run behavior via the entropy rate of various Markov models of finite populations.

The entropy rate of a Markov process is a measure of long-run variability by combining the (long-run) stationary distribution of the process with the (short-term) entropies of the transition probability distributions at each state of the process \cite{cover2006elements} \cite{ciuperca2005estimation}. Markov processes are studied intensely in biological applications, such as the Moran and Wright-Fisher processes \cite{moran1958random} \cite{moran1962statistical} \cite{ewens2004mathematical} \cite{nowak2006evolutionary}, which combine three fundamental biological process: stochastic drift, natural selection, and mutation. This paper contains several results on the entropy rate of the Moran process and generalizations, illuminating the relationships among the underlying fundamental processes on population dynamics in finite populations, such as the limit of small mutation probabilities, and the limit of large population sizes (eliminating drift), and the neutral fitness landscape (eliminating selection). Via these Markov processes, the entropy rate can be interpreted as an indicator of long-run biodiversity -- a large entropy rate indicates the likely long-term coexistence of multiple types in the population, whereas a vanishing entropy rate indicates the likely fixation of one or more of the population types. Results for populations of two-types were given in \cite{harper2013inherent}, some of which are extended now to arbitrarily many types.

In particular, for generalizations of the Moran process on arbitrarily many types, we show there is a tight upper bound on the entropy rate due to the structure of the transition probabilities of the process and that the neutral fitness landscape achieves this bound for fixed mutation rate and large population size. Typically small mutation leads to small entropy rates, other parameters held constant, meaning there is little variation in the long-term outcomes of the process (and fixation of one or more types). For the neutral landscape, large population sizes tend to lead to large entropy rates, for fixed mutation rates. We also give some computational results for a variety of multi-type landscapes exhibiting interesting population dynamics for which there are no known closed-form stationary distributions , such as the rock-scissors-paper landscape.

In addition to being a measure of long-run variation, and its intrinsic theoretical importance, the entropy rate of a process can act as an invariant (and generalizations of the entropy rate have been used for this purpose in dynamical systems theory, e.g. \cite{kolmogorov1959entropy} \cite{yag1959notion}). Loosely speaking, two processes producing similar long-run outcomes should have similar entropy rates (but not vice versa necessarily). A process with an entropy rate relatively close to zero, as we will see, is typically concentrated on the boundary of the simplex and is unlikely to have an interior stable state, in the sense of stationary stability or evolutionary stability \cite{harper2012stability}. Although we pursue analytical results in this paper (and give computational examples), it is possible to estimate the entropy rate of a Markov chain from sample data with an unbiased and asymptotically normal estimator \cite{ciuperca2005estimation}. (See also \cite{strelioff2007inferring}.) Hence as a practical matter, a set of population states from a real population could be used to estimate the entropy rate and make inferences on the long-run outcomes of the evolutionary process. 

Analytically, the processes under study are not generally reversible and so stationary distributions are often not known in closed form, but are easily computable. Two-type birth-death processes are reversible and there is somewhat useful formula for the stationary distribution; for the neutral landscape there exists a closed form for all dimensions \cite{khare2009rates}. These closed-forms allow analytic exploration of the behavior of the entropy rate in the absence of selection, particular in the relationship between population size and mutation rate. This yields, for instance, an explanation of the commonality of the relationship $N \mu = \text{constant}$ as an important arbiter of behavior in population models \cite{wakeley2005limits}.

\section{Preliminaries}

\subsection{Incentive Proportionate Selection with Mutation}

The incentive process was briefly introduced in \cite{harper2013inherent} as a generalization of the Moran process incorporating the concept of an incentive, introduced by Fryer in \cite{fryer2012existence}. An incentive is a function that mediates the interaction of the population with the fitness landscape, describing the mechanisms of many common dynamics via a functional parameter, including replicator and projection dynamics, and logit and Fermi processes. A Fermi incentive is frequently used to avoid the general issue of dividing by zero when computing fitness proportionate selection \cite{traulsen2007pairwise} \cite{traulsen2009stochastic}. The author described the fixation probabilities of the incentive process without mutation in \cite{harper2013incentive}. We now describe the incentive process with mutation, which captures a variety of existing processes, such as those used in \cite{claussen2005non} and \cite{taylor2004evolutionary}. This form was first described in \cite{harper2013stationary}. Many other authors study the Moran process with mutation, e.g. \cite{fudenberg2004stochastic}.

Let a population be composed of $n$ types $A_1, \ldots A_n$ of size $N$ with $a_i$ individuals of type $A_i$ so that $N = a_1 + \cdots + a_n$. Denote a population state by the tuple $a = (a_1, \ldots, a_n)$ and the population distribution by $\bar{a} = a / N$. We describe mutation by a matrix of mutations $M$ where $0 \leq M_{i j} \leq 1$ may be a function of the population state for our most general results, but we will typically assume in examples that for some constant value $\mu$, the mutation matrix takes the form $M_{i j} = \mu / (n-1)$ for $i \neq j$ and $M_{i i} = 1 - \mu$. Finally we assume that we have a function $\incentive(\bar{a}) = (\incentive_1(\bar{a}), \ldots, \incentive_n(\bar{a}))$ which takes the place of the quantities $a_i f_i(a)$ in the Moran process. Denote the normalized distribution derived from the incentive function as $\bar{\incentive}(\bar{a})$. To define the adjacent population states, let $i_{j k}$ be the vector that is 1 at index $j$, -1 at index $k$, and zero otherwise, with the convention that $i_{j j}$ is the zero vector of length $n$. Every adjacent state of state $a$ is of the form $a + i_{j k}$ for some $1 \leq j,k \leq n$. Assume that $N > n$.

The incentive process is a Markov process on the population states defined by the following transition probabilities, corresponding to a birth-death process where birth is incentive-proportionate with mutation and death is uniformly random. At a population state $a$ we choose an individual of type $A_i$ to reproduce proportionally to its incentive, allowing for mutation of the new individual as given by the mutation probabilities. The distribution of incentive proportionate selection probabilities is given by $p(\bar{a}) = M(\bar{a}) \bar{\incentive}(\bar{a})$; explicitly, the $i$-th component is
\begin{equation}
p_i(\bar{a}) = \frac{\sum_{k=1}^{n}{\incentive_k(\bar{a}) M_{i k} }}{\sum_{k=1}^{n}{\incentive_k(\bar{a})}}
\label{incentive-proportionate-reproduction} 
\end{equation}
The process also randomly chooses an individual to be replaced, just as in the Moran process. This yields the transition probabilities
\begin{align}
T_{a}^{a + i_{j k}} &= p_{j}(\bar{a}) \bar{a}_{k} \qquad \text{for $j \neq k$} \notag \\
T_{a}^{a} &= 1 - \sum_{b \text{ adj } a, b\neq a}{T_{a}^{b}}
\label{incentive_process}
\end{align}

We mainly consider incentives that are based on linear fitness landscapes of the form $f(\bar{a}) = A \bar{a}$ for a game matrix $A$. There are two one-parameter family of incentives, the $q$-replicator, given by $\incentive_i(\bar{a}) = \bar{a}_i^{q} f_i(\bar{a})$, and the $q$-Fermi, given by $\displaystyle{ \incentive_i(\bar{a}) = \frac{\bar{a}_i^{q} \text{exp}(\beta f_i(\bar{a}))}{\sum_{j}{\bar{a}_j^{q}\text{exp}(\beta f_j(\bar{a}))}}}$ for $\beta >0$. Note that the incentive need not depend on a fitness landscape or the population state at all. The Fermi process of Traulsen et al is the $q$-Fermi for $q=1$ \cite{traulsen2009stochastic}, and $q=0$ is often called the logit incentive, which is used in e.g. \cite{andrae2010entropy}. The $q$-replicator incentive has previously been studied in the context of evolutionary game theory \cite{harper2011escort} \cite{harper2012stability} \cite{harper2013incentive} and derives from statistical-thermodynamic and information-theoretic quantities \cite{tsallis1988possible}. The Fermi incentive is particularly convenient for populations with more than two types since the mean-fitness can otherwise be zero in the denominator (e.g. for zero-sum games such as the rock-paper-scissors game), which would cause the transition probabilities to be ill-defined. We assume that the fitness landscape and mutation matrix are such that the process has a stationary distribution, i.e. that the process has no absorbing states.

\subsection{Stationary Distributions}

Stationary distributions for the Moran process in two dimensions and some recently-studied generalizations are given by Claussen and Traulsen in \cite{claussen2005non} (see also \cite{antal2009strategy} \cite{taylor2004evolutionary} \cite{khare2009rates} \cite{gokhale2011strategy} \cite{wu2013dynamic}). Explicit formulas for stationary distributions of Moran processes with more than two types are not known since the process is not generally reversible \cite{ewens2004mathematical}. For the case of neutral fitness, however, the process is reversible. In \cite{khare2009rates}, an explicit formula for the stationary distribution Moran process for neutral fitness is given (Proposition 4.7, page 20). Let the rising factorial be defined as
\[ (x)_y = x (x+1) \cdots (x+y-1),\]
and note that $(x)_{0}$ is 1 by definition. The stationary distribution of the neutral landscape is given by a Dirichlet-multinomial distribution, in our notation, as
\begin{equation}
s_a = \binom{N}{a} \frac{\prod_{i=1}^{n}{(\alpha)_{a_i}}}{(n \alpha)_{N}},
\label{dirichlet_multinomial}
\end{equation}
where $\alpha = \frac{N \mu}{n-1 -n\mu}$, for $\mu \neq (n-1)/n$. This formula can be directly shown to satisfy the detailed-balance condition \cite{taylor2006symmetry}, verifying its validity and the reversibility of the process. For $\mu = (n-1)/n$, the stationary distribution is $s_a = \binom{N}{a}n^{-N}$ (for any fitness landscape). Note that throughout the paper we assume that when $\mu \neq 0$, the incentive process is irreducible and has a stationary distribution.

\subsection{Entropy Rate}

A fundamental tool in information theory, probability, and statistics is the Shannon entropy of a probability distribution \cite{shannon1949mathematical}. For a discrete probability distribution $p = (p_0, p_1, \ldots, p_n)$, the Shannon entropy (or simply entropy) is
\[ H(p) = -\sum_{i=0}^{n}{ p_i \log{p_i}},\]
where $p_i \log{p_i} = 0$ if $p_i = 0$. The meaning of the entropy of a probability distribution is often described as a measure of uncertainty or information content. The entropy rate of a stationary Markov process is an information-theoretic quantity that characterizes the inherent randomness of the process \cite{cover2006elements} \cite{strelioff2007inferring}, and plays a similar role as the Shannon entropy. To each state $a$ of a Markov process $P$ there is a probability distribution $T_a = (T_{a}^{b})_{b \text{ adj } a}$ for the transition probabilities out of the state to adjacent states (there are $n(n-1)+1$ such states for the $n$-type incentive process). We refer to the entropies of these transition probability distributions as the transition entropies $H(T_a)$. The mean of the transition entropies taken with respect to the stationary distribution $s$ of the Markov process is the entropy rate:
\begin{align*}
\label{entropy_rate}
H(P) &= -\sum_{a}{s_a \sum_{b \text{ adj } a}{T_{a}^{b} \log T_{a}^{b}} } = \sum_{a}{s_a H(T_a) }
\end{align*}

The stationary distribution is a description of the long term behavior of a Markov process, and so the entropy can be similarly interpreted as a measure of the uncertainty, inherent randomness, or information content of the long run behavior of the process. The entropy rate is affected by the likelihood that the process occupies a particular state and the entropy of the behavior of the state. In other words, the entropy rate reflects both the long term variance in population states (the stationary distribution) and the short term variance due to the entropy of the transition probabilities at the states represented significantly in the stationary distribution.

\section{Entropy Rates for the Incentive Process}

We start with a collection of simple observations about values of the entropy rate. The first is that the transition entropy is positive at interior stationary states (where no $a_i = 0$). This can then be used to bound the value of the entropy rate if the stationary distribution has local or global extrema, which is very often the case \cite{harper2013stationary}.

\begin{proposition}
Let $s$ be the stationary distribution of an incentive process defined by Equation \ref{incentive-proportionate-reproduction}.
 \begin{enumerate}
  \item If $s_a > 0$ at any interior state $a$ then $H(P) > 0$
  \item In particular, if the process has an interior stationary stable state (a local maximum), $H(P) > 0$
  \item If $s$ has a global maximum at state $a$ then $H(P) \leq H(T_a)$
  \item If $s$ has a global minimum at state $a$ then $H(P) \geq H(T_a)$
\end{enumerate}
\end{proposition}
\begin{proof}
(2)-(4) are easy consequences of (1), which holds because at an interior state $a$ and state $b = a + i_{j, k}$ ($k \neq j$ so $a \neq b$) adjacent to a, $0 \leq T_{a}^{b} = p_j(\bar{a})\bar{a}_k < 1$. For $H(T_a) = 0$, we must have that all $p_j(\bar{a}) = 0$, but this is impossible since $\sum_{j}{p_j(\bar{a})} = 1$ and $0 \leq p_j \leq 1$.
\end{proof}

The entropy rate can be zero if the stationary distribution is nonzero only on the boundary states, i.e. if the only long-run outcomes are fixation on one or more of the population types. The following theorem shows that for the incentive process with some of the described incentives that as the mutation rate $\mu \to 0$ for fixed population size $N$ and number of types $n$, the entropy rate tends to zero. Intuitively this is a result of the fact that as $\mu \to 0$, the corner points of the simplex become absorbing states of the process, and so the stationary distribution collects on these points \cite{fudenberg2006imitation} \cite{fudenberg2006evolutionary} \cite{fudenberg2008monotone}. Biologically, we can interpret this theorem as indicating that long-term variation requires nonzero mutation for these incentives (however will we show that for the best-reply incentive, to which the theorem does not apply, there can be variation even when $\mu = 0$). In other words, drift and selection acting without mutation will clearly lead to fixation in finite populations (since there is no way to escape the absorbing states), and the amount of long-run variation is controllable (to some extent) by the mutation rate. This is in sharp contrast with the behavior of infinitely large populations (without drift or mutation) as dictated by the replicator dynamic, in which coexistence of multiple types can occur from the action of selection alone, e.g. for the landscape on two types given by $a=1=d$ and $b=2=c$ \cite{hofbauer2003evolutionary}. While this is not a novel idea by any means \cite{nowak2004emergence} \cite{nowak2006evolutionary}, Theorem \ref{theorem_mu_to_zero} shows that the entropy rate is sensitive to this behavior. To avoid repetition of technical hypotheses in  \cite{fudenberg2008monotone}, the theorem is only stated for the appropriate incentives defined in the previous section.

\begin{theorem}
Let $q \geq 1$. For any fixed $n$ and $N$, the entropy rate of the incentive processes with the $q$-replicator and the $q$-Fermi incentives on linear fitness landscapes tend to zero as $\mu \to 0$.
\label{theorem_mu_to_zero}
\end{theorem}
\begin{proof}
For these incentives, we can consider the incentive process as a special case of the frequency-dependent Moran process \cite{nowak2004emergence} with modified fitness landscapes given by $\bar{a_i}^{q-1} f_i(x)$. It is shown in \cite{fudenberg2008monotone} that as the mutation rates $\mu \to 0$, the stationary distribution tends to zero except for the boundary states, where the stationary distribution is concentrated. The entropy at the corner states is given solely in terms of $\mu$. Specifically, we have that the transition distribution is (up to reordering) $(1-\mu, \mu / (n-1), \ldots, \mu / (n-1))$, where there are $n$ terms overall. The entropy of these distributions is $-H = (1-\mu) \log{(1-\mu)} + \mu \log{(\mu / (n-1))}$ which also limits to zero as $\mu$ does. We have shown that every term in the entropy rate limits to zero.
\end{proof}

Theorem \ref{theorem_mu_to_zero} shows that the analogous result for $n=2$ for the Moran process in \cite{harper2013inherent} generalizes to all $n$ for the Moran and Fermi processes. The latter case is crucial for $n>2$ processes since cycling behavior can lead to average and total fitnesses equal to zero, which would be ill-defined for the Moran process. The theorem, however, does not apply to the case $q=0$, which for the $q$-replicator and $q$-Fermi incentives correspond to the projection and logit dynamics. These cases violate the assumption in \cite{fudenberg2008monotone} that absent strategies ($a_i=0$) will never be re-introduced when $\mu = 0$. Direct calculations indicate that the stationary distribution does not tend to the boundary for $q=0$, $n=2$, and the neutral fitness landscape (the stationary distribution can be computed exactly in these cases). We will give a detailed example for the best-reply incentive in the applications.

\subsection{Strict Upper Bound}

Having shown the lower bound of zero for the entropy rate in the limit of vanishing mutation probabilities, we now show that the entropy rate of the incentive process is bounded above and that the upper bound is realized in the limit of large population size. The case for dimension $n=2$ was shown previously in \cite{harper2013inherent}. We prove a more general result for all Markov processes of a particular form, which includes some P\'olya urn models (see \cite{khare2009rates}).

\begin{theorem}
Let $x = (x_1, \ldots, x_n)$ and $y = (y_1, \ldots, y_n)$ be discrete probability distributions. Define a discrete probability distribution $z$ by $z_{i j} = x_i y_j$ for $1 \leq i,j \leq n$, $i \neq j$, and $z_0 = 1 - \sum_{i,j}{z_{ij}}$. (There are $n(n-1) + 1$ terms of $z$.) Then the maximum entropy of $z$ occurs when $x=y$ (for either fixed) and $z$ is globally maximized when $x_i=1/n=y_i$ for all $i$. This maximum is
\[ H(z) \leq \frac{2n-1}{n} \log{n}\]
\label{incentive_bound}
\end{theorem}
\begin{proof}
Intuitively, the maximum entropy should occur when $x_i = \frac{1}{n} = y_i$ for all $1 \leq i \leq n$. This choice of $x$ and $y$ gives $z_{i j} = 1/n^2$ and $z_0 = 1/n$, which has entropy given by the bound above. To prove that this is in fact the maximum possible entropy, use Lagrange multipliers with the constraints that $x_1 + \ldots + x_n = 1$, $y_1 + \ldots + y_n = 1$. A tedious but straightforward calculation yields $x_i = y_i$ for all $i$, which then leads to $x_i = 1/n$ for all $i$.
\end{proof}

Since we have that the transition probabilities for the incentive process are of the form defined in Theorem \ref{incentive_bound}, and the entropy rate is bounded by the largest entropy of the transition distributions, we have the following corollary, which applies to any incentive and mutation matrix (so long as the stationary distribution exists).

\begin{corollary}
The maximum entropy rate for the incentive process for any incentive and any mutation matrix on $n$ types is 
\[ H(z) \leq \frac{2n-1}{n} \log{n}.\]
\end{corollary}

The case for the incentive process for $n=2$ was done in the appendix of \cite{harper2013inherent}, in which the explicit value $(3/2) \log{2}$ was found to be the maximum.  The maximum attainable entropy for a probability distribution of length $n(n-1) +1$ is simply $\log{(n(n-1) + 1)}$, (e.g. if the stationary distribution and all transition distributions were uniform), and so this also bounds the entropy rate of Markov processes with transition distributions of the same length. Theorem \ref{incentive_bound} shows that for the incentive process, the bound is less than the theoretical bound for processes on the same state space. The proportion of the theoretical maximum attainable, as a function of $n$, is
\[ \frac{2n - 1}{n} \frac{\log{n}}{\log{(n (n-1) + 1)}} \to 1,\]
as $n \to \infty$. For $n=2$, the proportion is already $\approx 94\%$. While the bound is lower than the maximum for any finite $n$, it is not much lower. The reason for the bounding is simply that the transition probabilities, by their definition, cannot all be exactly equal to $(1/n, \ldots, 1/n)$.

\subsection{Neutral Selection}

In this section we consider only the neutral fitness landscape for the Moran and Fermi processes (which are equal in this case). The neutral landscape is a useful as a reference case and for eliminating the action of selection, and has an available stationary distribution for all $n$. We start with the following corollary of Theorem \ref{incentive_bound}, which allows us to show that the entropy rate is monotonic in $\mu$. This demonstrates that for fixed drift and no selection, the long-run behavior is controlled by the mutation rate. Intuitively, the long-run variation is monotonic in the mutation rate.

\begin{corollary}
Suppose that $n$ divides $N$. For any $\mu$, the transition entropy is largest at the central state $(N/n, \ldots, N/n)$. If $n$ does not divide $N$, then the entropy is maximal at any permutation of $(c,c,\ldots, N - (n-1)c)$ with $c = \lfloor \frac{N}{n} \rfloor$, i.e. the central states.
\label{max_transition_entropy}
\end{corollary}

\begin{proof}
By Theorem \ref{incentive_bound}, the entropy of the transition distribution at a state $\bar{a}$ is maximal when $\bar{a}_i (1 - \mu) + \mu / (n-1) (1-\bar{a}_i) = \bar{a}_i$, and so $a_i = N/n$ for all $i$, the central state.
\end{proof}

\begin{theorem}
For $n$ types and population size $N$ fixed, the entropy rate of the neutral landscape is monotonically increasing in $\mu$.
\label{theorem_mononotic_mu}
\end{theorem}

\begin{proof}
Since the transition entropies are symmetric and decreasing outward from the central state, we need to maximize the concentration of the stationary distribution on the center of the simplex. As $\mu$ increases, the stationary distribution is more highly concentrated on the central states. This is because the stationary distribution is symmetric and maximal at the central state, and decreasing toward the boundary states. We have that the derivative with respect to $\mu$ of the logarithm of the stationary distribution at the central state is proportional to 
\[ \sum_{i=0}^{\lfloor \frac{N}{n} \rfloor -1}{\frac{1}{\alpha + i}} \frac{\partial \alpha}{\partial \mu} = \sum_{i=0}^{\lfloor \frac{N}{n} \rfloor-1}{\frac{1}{\alpha + i}} \frac{N(n-1)}{(n-1 -n\mu)^2} > 0.\]
\end{proof}

It was shown in \cite{harper2013inherent} that for the Moran process and the neutral fitness landscape, the entropy rate approaches the bound in the large population limit. This is the case for higher dimensions as well, and shows that in the absence of selection and fixed mutation rate, drift controls the behavior of the population.

\begin{theorem}
Define $E(N,n)$ to be the entropy rate of the Moran and Fermi processes ($q=1$) on the neutral fitness landscape, with population size $N$ and mutation probability $\mu$ fixed. Then for as $N \to \infty$ the entropy rate approaches the maximum:
\[ \lim_{N \to \infty}{E(N, n)} = \frac{2n-1}{n} \log{n}.\]
\end{theorem}
\begin{proof}
Assume without loss of generality that $N$ is divisible by $n$. Let $\mu \neq (n-1)/n$ and let $b = (N/n, \ldots, N/n)$ be the central state. The multinomial coefficient $\binom{N}{a}$ is largest when $a=b$, and in that case, by Sterling's approximation, we have that $\binom{N}{b} \approx n^N$. It is easy to see that this term dominates the stationary distribution for large $N$ (e.g. by looking at the logarithms as before), so we have that $H(P) \to H(T_b)$ as $N \to \infty$. $H(T_b)$ is precisely the bound given in Theorem \ref{incentive_bound}. The case for $\mu = (n-1)/n$ is similar.
\end{proof}

Does any other landscape yield a larger entropy rate for any $N$?  We can exceed the entropy rate of the neutral landscape with landscapes that more tightly hold the stationary distribution about the central state. For instance, for the fitness landscape defined by the matrix $a=0=d$, $b=1=c$ ($n=2$), the stationary distribution is essentially binomial and the relationship between $\mu$ and $N$ at the central state is $\mu 2^N$ \cite{claussen2005non} \cite{harper2013inherent}. A direct computation gives an entropy rate of $\approx 1.015$ for the former and $ \approx 0.819$ for the latter (the neutral case). Direct calculation also shows that for $q=0$ and $q=1/2$ the entropy rate for the neutral landscape is greater than that of the $q=1$ case. These can be thought of as nonlinear fitness landscapes for the Moran process. 


\section{Examples and Applications}

\subsection{Application: Population Size and Mutation Rate}

We have seen now that the entropy rate for the neutral landscape depends heavily on the behavior of the composite parameter $\alpha$ and its dependence on $\mu$ and $N$. Let us examine the behavior of the entropy rate of the neutral landscape when the population size and mutation rate are dependent. Let $\mu = \frac{n-1}{n} \frac{1}{N^k}$ for a scaling parameter $k$, then $\alpha = \frac{N}{n} \frac{1}{N^k - 1}$. Holding $n$ fixed, as the population size $N$ becomes large the value of $\alpha$ depends strongly on scaling parameter: for $k<1$, $\alpha \to \infty$ and so the entropy rate tends to the maximum, for $k > 1$, $\alpha \to 0$  and so does the entropy rate, and for $k=1$, $\alpha \to 1 / n$. For $\mu = \frac{n-1}{n} \frac{1}{N+1}$, $\alpha = 1/n$ for all $N$, so we briefly alter $\mu$ for convenience. Since $\alpha$ is constant with respect to $N$, the value of the stationary distribution at the central state is decreasing (as the probability spreads throughout the simplex); the transition entropy at the central state is increasing in $N$ to the maximum. The end result is that the entropy rate limits quickly to a finite value less than the maximum (depending on $n$) as the stationary distribution is nearly uniform away from the boundary. See Figure \ref{figure_1}.

\begin{figure}[h!]
    \centering
    \includegraphics[width=0.5\textwidth]{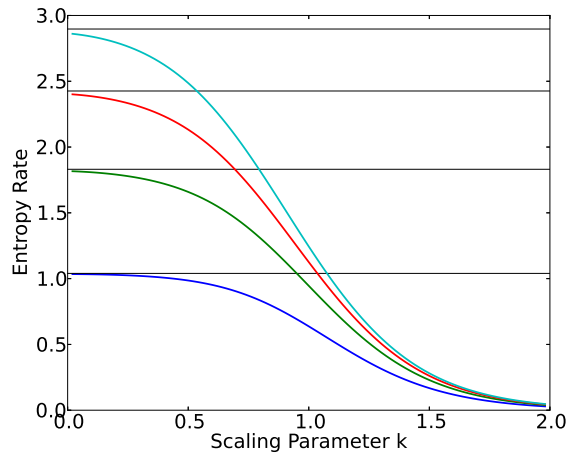}
    \caption{Entropy rate vs. scaling parameter $k$ for the neutral landscape, $n=2,3,4,5$ (bottom to top), $N=50$, $\mu = \frac{n-1}{n} \frac{1}{(N+1)^k}$. Horizontal lines indicate the maximum possible entropy rate for each $n$. Note that in accordance with Theorem \ref{theorem_mononotic_mu}, the entropy rate is monotonically decreasing with increasing $\mu$ (decreasing $k$).}
    \label{figure_1}
\end{figure} 

\subsection{Application: Strong and Weak Selection}

Consider an incentive process with the replicator or Fermi incentive with strength of selection $\beta$ and the traditional Moran landscape on two types given by $a=b=r$, $c=1=d$. For these landscapes the first type has a unilateral advantage over the second type (when $r > 1$), and the stationary distributions are skewed toward the boundary state $(N, 0)$, reducing the entropy rate (demonstrated computationally in \cite{harper2013inherent}). In the $n=2$ case the stationary distribution can be computed exactly, and we find that the entropy rate is maximal when $\beta \to 0$ or $r \to 1$, and monotonically decreasing away from $r=1$ and for $\beta > 0$, so the entropy rate respects the weak selection limit (see Figure \ref{figure_2}). Similarly, for the strong selection limit $\beta \to \infty$ we have for the Moran landscape that the process skews toward $(N,0)$, which becomes absorbing, and so the entropy rate drops to zero. Similarly for $r \to 0$ or $r \to \infty$. 
This is a very intuitive behavior for the entropy rate. As the strength of selection increases, the variation in long term outcomes should decrease, and for the Moran landscape the variation vanishes. For weak selection, the population dynamics is dominated by drift and mutation, and has some long-run variation outside of extreme cases (such as $\mu=0$). 

For other fitness landscapes, strong selection does not necessarily eliminate the entropy rate. For instance, for the landscape given by $a=1=d$, $b=2=d$, (which for the replicator dynamic is evolutionarily stable at the central point), for large $\beta$ the Fermi incentive tends to the best-reply incentive, which has stationary distribution that is strongly concentrated on the central point, and produces a nonzero entropy rate as $\beta \to \infty$. The best reply incentive is $(a_1 / N, 0)$ if type $A_1$ has higher fitness, $(0, a_2 / N)$ if type $A_2$ has higher fitness, and $(a_1 / N, a_2 / N)$ if the types have equal fitness. The value of the entropy rate is still dependent on the population size $N$ and mutation rate $\mu$. More specifically, for any $\beta$ the stationary distribution is maximal at the central point, and increasingly so for larger $\beta$ (and larger $N$). The entropy of the transition distribution, however, decreases from $H(\frac{1}{2}, \frac{1}{4}, \frac{1}{4}) = \frac{3}{2} \log{2}$ for $\beta=0$ to $H(\frac{\mu}{2}, \frac{1-\mu}{2}, \frac{1}{2})$ for very large $\beta$. This gives a nonzero entropy rate that is decreasing in $\beta$ but does not vanish completely even if $\mu \to 0$, in which case the entropy is $\log{2}$. (Note that Theorem \ref{theorem_mu_to_zero} does not apply to the best reply incentive.)

If the mutation rate is zero, why is there still variation? The best reply incentive at the central point is $(1/2, 1/2)$, and so the population will move outward from the central state in either direction with probability $1/2$. Then population then moves right back to the center with probability 1. So the uncertainty in the population state is entirely do to the coin-flip activity at the central state, which is exactly $\log 2$, the entropy rate.


\begin{figure}[h!]
    \centering
    \includegraphics[width=0.5\textwidth]{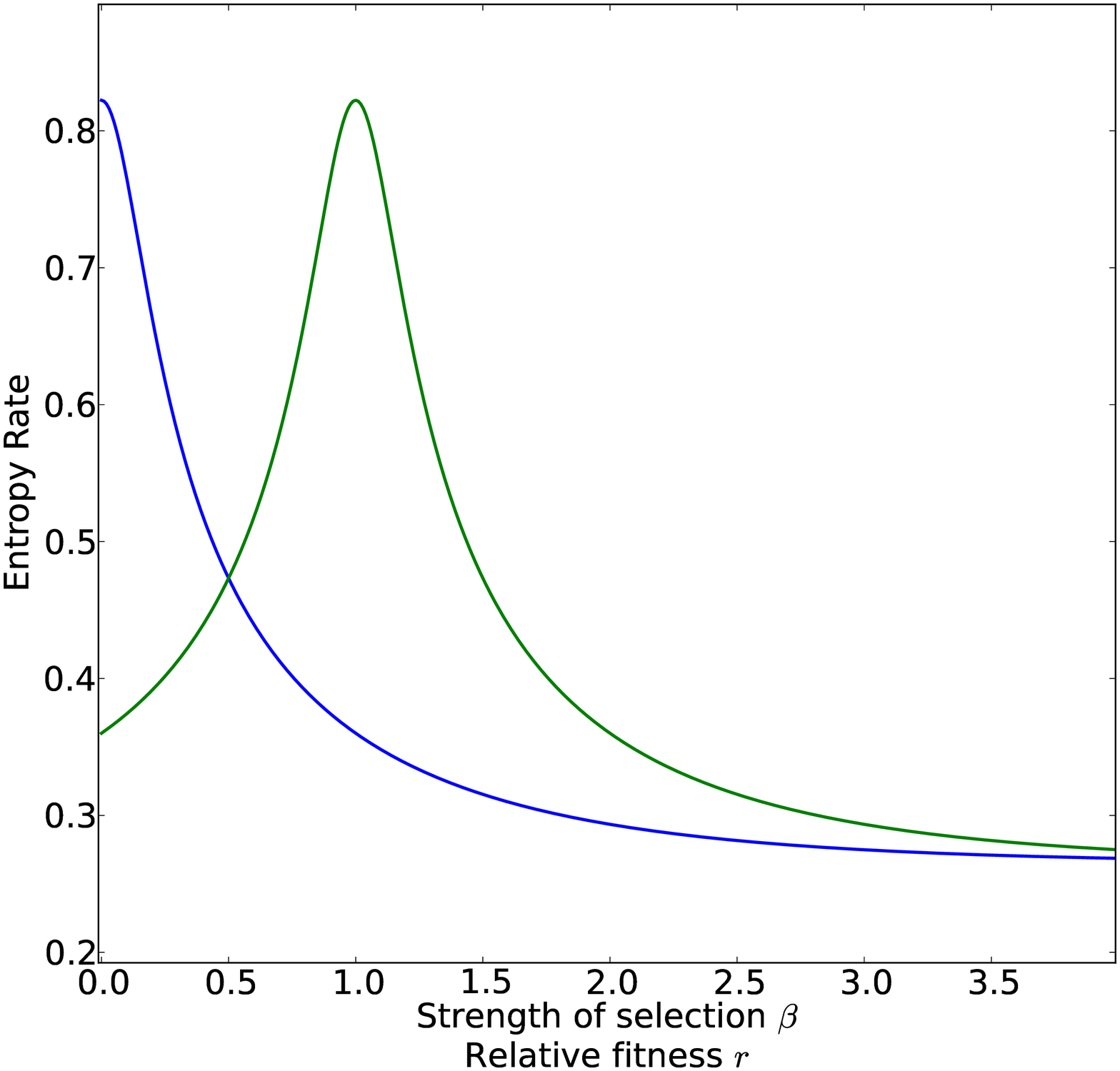}
    \caption{Weak selection limit: Entropy rate vs. Strength of selection / relative fitness, $N=30$, $n=2$, $\mu=1/N$. Blue curve: The process is defined by Fermi incentive on the Moran landscape $r=2$. The entropy rate is monotonically decreasing in $\beta$ and maximal at $\beta = 0$ where the landscape is neutral. Green curve: The process is defined by Fermi incentive with $\beta=1$ on the Moran landscape for variable $r$. The entropy rate is maximal when $r=1$ and decreasing away from $r=1$.
    }
    \label{figure_2}
\end{figure} 

\subsection{Computational Examples for Some Three-type Landscapes}

Selection can reduce the entropy rate, as in the examples of the previous section, but it can also increase it by more highly concentrating the stationary distribution in the interior of the simplex. Intuitively, if fitness is frequency independent, selection will favor one or more types and the stationary distribution tends to concentrate on the boundary states, increasing the entropy rate. For frequency dependent selection, the situation is more complex as the entropy rate could increase or decrease. Let us consider some specific landscapes for $n=3$ using the Fermi incentive with $\beta=1$.

I.M. Bomze's gives a classification of 49 three-type phase portraits for the replicator dynamic with linear fitness landscape (see \cite{bomze1983lotka} and the additions and corrections \cite{bomze1995lotka}). These landscapes catalog a variety of behaviors including interior equilibria, connected sets of interior equilibria, and boundary equilibria. For $N=30$, $\mu=1/N$, $\beta=1$, the entropy rate of the neutral landscape (Bomze's \#0) has entropy rate $1.155$; 27 of the remaining 48 matrices given by Bomze have entropy rate less than the neutral landscape, with an overall minimum of 0.38 for \#13 and a maximum of 1.62 for \#7. The relationships between selection, drift, and mutation are more complex for three types and this is reflected in the entropy rate. The rock-paper-scissors landscape (\#16) is nearly the same as the neutral landscape (1.152). A generalized rock-paper-scissors matrix takes the form
\[ A = \left( \begin{smallmatrix}
     0  & -b & a \\
     a  & 0  & -b\\
     -b & a  &  0\\
    \end{smallmatrix} \right) \]
Bomze's \#16 is the case $a=1=b$. Fix $b=1$. For the replicator dynamic, the central point is stable for $a>1$ and unstable for $a<1$. Computations indicate that the entropy rate (with all other parameters fixed) is increasing in $a$ as central state switches from unstable to stable and the stationary distribution shifts from concentration on the boundary to the central state, as expected. See Figure \ref{figure_3}. The behavior with respect to increasing $N$ depends on whether $a > b$ or not.

\begin{figure}[h!]
    \centering
    \includegraphics[width=0.4\textwidth]{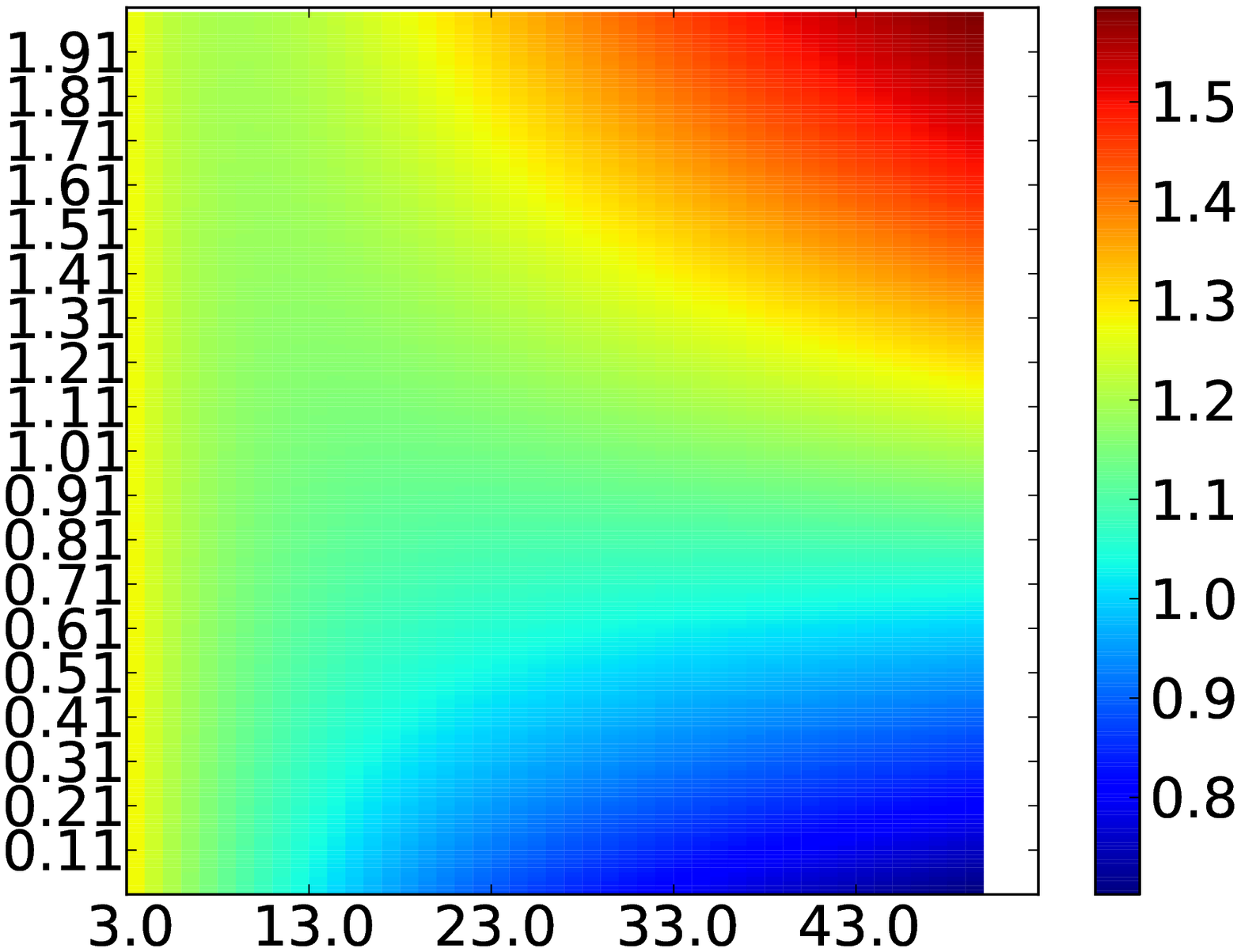}
    \includegraphics[width=0.4\textwidth]{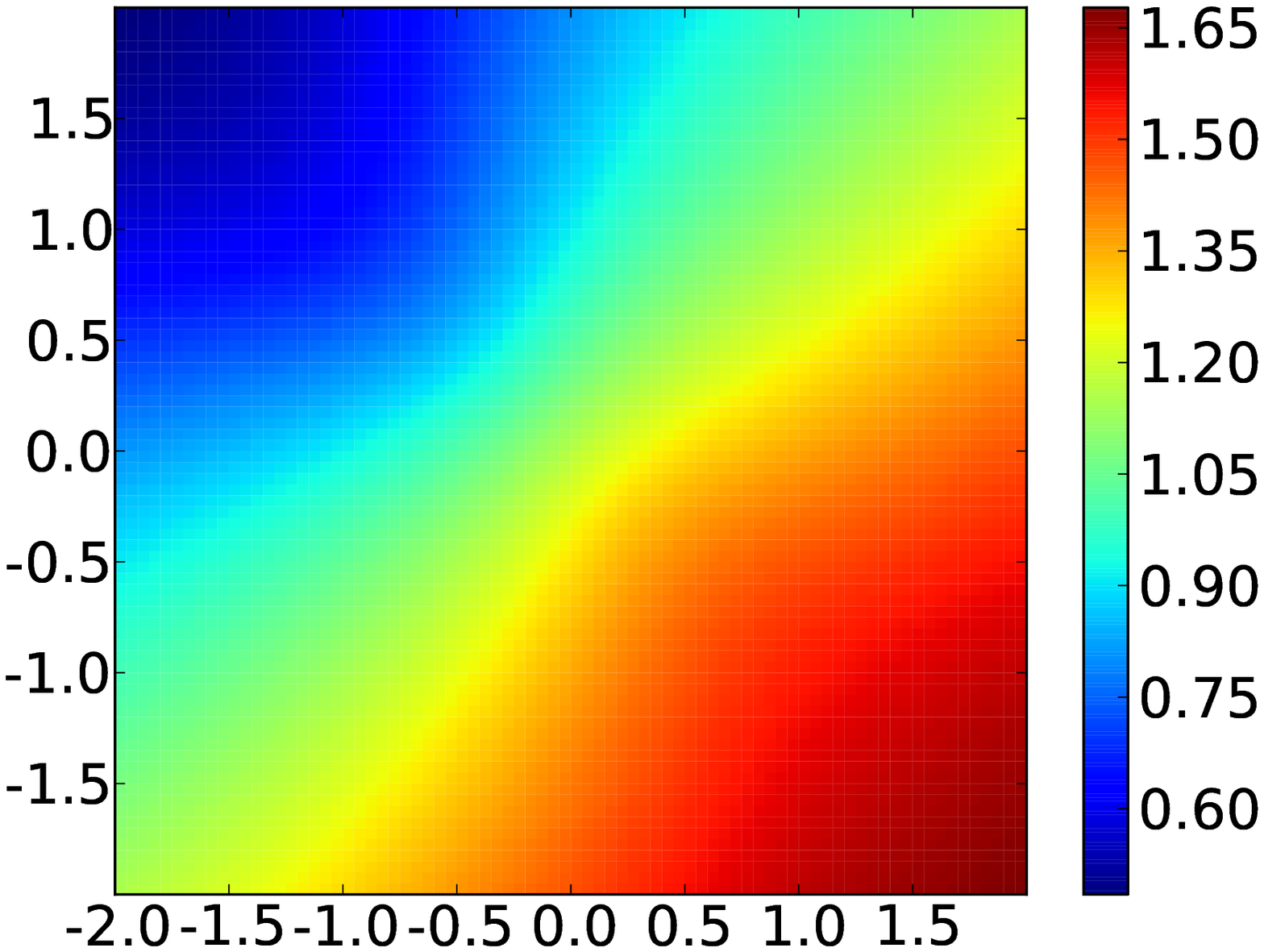}
    \caption{Left: Entropy rates for population size $3 \leq N < 50$ (horizontal axis) and Rock-Scissors-Paper parameter $a$ (vertical axis), with $b=1$ fixed, and Fermi selection given by $\beta=1$, and $\mu = 1/N$. For $a > b$, the population tends toward the center of the simplex, resulting in a higher entropy rate. For $a<b$, the stationary distribution builds up on the boundary, and is decreasing as the population size increases. Right: Entropy rates for RSP with $-2<a<2$ on the horizontal axis and $-2<b<2$ on the vertical, for $N=20$, $\beta=1$.}
    \label{figure_3}
\end{figure}

\section{Discussion}

We have seen in multiple contexts that the entropy rate scales monotonically with mutation for the most commonly used selection incentives (Moran and Fermi) for the neutral landscape, vanishes with mutation, approaches an upper bound (increasing in the number of types) for large populations, and scales monotonically with respect to strength of selection and relative fitness. As such, the entropy rate is an intuitive invariant for generalized birth-death processes, and useful as a means of comparing such processes and their dependences on fundamental evolutionary parameters. Moreover, we have seen that the entropy rate explains at least some commonly observed scalings between evolutionary parameters, such as the relationship between drift and mutation. The interested reader should consult \cite{harper2013inherent} for a large variety of computational and analytic examples for two-type populations, which revel additional relationships between relative fitness, mutation, and population size.

The author proposes entropy rate as a useful characterization not just of the long run variation but also of the relative stability of populations in which the equilibrium points are the same \cite{harper2013stationary}. For instance we have seen an example of a population evolving under a best reply dynamic that has an equilibrium (most likely stationary state) at the central state $(N/2, N/2)$. A population evolving under fitness proportionate selection with either a neutral landscape or a hawk-dove landscape ($a = 1 = d$, $2 = b = c$) will have the same maximum of the stationary distribution, but will have more variation in the stationary distribution. Moreover, the entropy rates of these processes generally differ whereas other long run comparison tools like stationary abundance \cite{antal2009strategy} \cite{gokhale2011strategy} \cite{allen2012measures} \cite{wu2013dynamic} may report the same expectations (since the stationary distributions are symmetric). Hence entropy rate can sometimes distinguish between these processes in a way that other common methods cannot. Furthermore, since entropy rates can be estimated from data, they may be the basis of statistical tests for neutrality. Stationary abundance is able to distinguish cases that the entropy rate cannot, for instance in cases in which $\mu \to 0$ and the stationary distribution (only nonzero on the boundary) is not symmetric, since the entropy rate is zero in all such cases. Therefore a measure of long term variation like entropy rate is a complementary invariant to other tools in use for comparing the long run outcomes of evolving populations.

\subsection*{Methods}
All computations were performed with python code available at \url{https://github.com/marcharper/stationary}. This package can compute exact stationary distributions for detail-balanced processes and approximate solutions for all other cases mentioned in this manuscript. All plots created with \emph{matplotlib} \cite{Hunter:2007}.

\subsection*{Acknowledgments}

This research was partially supported by the Office of Science (BER), U. S. Department of Energy, Cooperative Agreement No. DE-FC02-02ER63421.

%
%
%
%

\bibliography{ref}
\bibliographystyle{plain}

\end{document}